\documentclass[reqno, 10pt]{amsart}
\usepackage{a4wide}
\usepackage[greek, russian, german, french, italian, english]{babel}

\usepackage[T3, T1]{fontenc}
\usepackage[utf8]{inputenc}
\usepackage[italian]{varioref}

\usepackage{CJKutf8}

\usepackage{todonotes}

\usepackage{titlesec}
\usepackage{titletoc}

\renewcommand{\footnotesize}{\scriptsize}
\renewcommand{\thefootnote}{\alph{footnote}}

\usepackage[perpage]{footmisc}

\newcommand{\symbolfootnote}[1]{%
\let\oldthefootnote=\thefootnote%
\stepcounter{mpfootnote}%
\addtocounter{footnote}{-1}%
\renewcommand{\thefootnote}{\fnsymbol{mpfootnote}}%
\footnote{#1}%
\let\thefootnote=\oldthefootnote%
}

\usepackage{emptypage}
\usepackage{microtype} 
\usepackage{fancyhdr}
\usepackage{indentfirst}
\usepackage{changepage}

\usepackage{empheq} 

\usepackage{tikz-cd}
\usetikzlibrary{decorations.fractals}
\usetikzlibrary{lindenmayersystems} 

\usepackage{booktabs} 
\usepackage{tabularx} 
\usepackage{longtable}
\usepackage{graphicx} 
\usepackage{pict2e, curve2e}
\usepackage{listings}
\usepackage{enumitem}

\usepackage[all]{xy}

\usepackage{caption}
\usepackage{xcolor}
\definecolor{eggplant}{HTML}{800080}
\definecolor{mallard}{HTML}{008080}
\definecolor{dusky cerulean}{HTML}{004080} 
\definecolor{Chinese loquat}{HTML}{F7C015}
\definecolor{pumpkin}{HTML}{FF7518}
\definecolor{cyan-blue}{HTML}{18A2FF}
\definecolor{vernal green}{HTML}{03E364}
\definecolor{magenta}{HTML}{E30382}
\definecolor{cynara violet}{HTML}{8803E3}

\usepackage{amsfonts}
\usepackage{amsmath}
\usepackage{amssymb}
\usepackage{amsaddr}
\usepackage{amsthm}
\usepackage{mathrsfs}
\usepackage[scr]{rsfso} 
\usepackage{esstixfrak}
\usepackage{eufrak}
\usepackage{fixmath}

\usepackage{scalerel}
\newlength\bshft
\def\pseudobold#1{\ThisStyle{\ooalign{$\SavedStyle#1$\cr%
  \kern-\bshft$\SavedStyle#1$\cr%
  \kern\bshft$\SavedStyle#1$}}}
  
\usepackage[version=4]{mhchem}

\usepackage[bbgreekl]{mathbbol}
\DeclareSymbolFontAlphabet{\mathbb}{AMSb}
\DeclareSymbolFontAlphabet{\mathbbl}{bbold}
                        
\usepackage{mathtools}
\usepackage{latexsym}
\usepackage{siunitx}
\usepackage{icomma}
\usepackage{slashed}
\usepackage{braket}
\usepackage{xfrac} 

\usepackage{bbding}

\usepackage{textcomp}

\titleformat{\section}[block]
{\large
\bfseries}
{\thesection.}{0.5em}{}

\titleformat{\subsection}[block]
{\normalsize
\bfseries}
{\thesubsection.}{0.5em}{}

\titleformat{\subsubsection}[block]
{\normalsize
\bfseries}
{\thesubsubsection.}{0.5em}{}

\contentsmargin{2.65em}

\titlecontents{chapter}
				[0pc]
				{\addvspace{.8pc}}
				{\bfseries\thecontentslabel{. }}
				{\hspace*{0em}}
				{\hfill\contentspage}
				[\addvspace{.3pc}]
				
\titlecontents{section}
				[0pc]
				{}
				{\thecontentslabel{. }}
				{\hspace*{0em}}
				{\titlerule*[0.4pc]{\thinspace.}\contentspage}

\titlecontents{subsection}
				[1.65pc]
				{}
				{\thecontentslabel{. }}
				{\hspace*{0em}}
				{\titlerule*[0.4pc]{\thinspace.}\contentspage}

\titlecontents{subsubsection}
				[3.95pc]
				{}
				{\thecontentslabel{. }}
				{\hspace*{0em}}
				{\titlerule*[0.4pc]{\thinspace.}\contentspage}

\newtheoremstyle{personal1}
  {\topsep}   
  {\topsep}   
  {\upshape}  
  {0pt}       
  {\itshape}  
  {.}         
  {5pt plus 1pt minus 1pt} 
  {\bfseries\thmname{#1}\thmnumber{ #2}\normalfont\thmnote{ \itshape(#3)}}
\theoremstyle{personal1}

\newtheoremstyle{personal2}
  {\topsep}   
  {\topsep}   
  {\itshape} 
  {0pt}      
  {\itshape} 
  {.}       
  {5pt plus 1pt minus 1pt} 
  {\bfseries\thmname{#1}\thmnumber{ #2}\normalfont\thmnote{ \itshape(#3)}} 
\theoremstyle{personal2}

\newtheorem{theorema}{Theorem}[section]

\newtheoremstyle{personal3}
  {\topsep}   
  {\topsep}   
  {\upshape}  
  {0pt}       
  {\itshape} 
  {.}         
  {5pt plus 1pt minus 1pt} 
  {\thmname{#1}\thmnumber{ \itshape#2}\thmnote{ (#3)}} 
\theoremstyle{personal3}

\usepackage{float}

\graphicspath{figures}

\usepackage{url}

\usepackage[pdfauthor = {Edoardo{ }Niccolai},
	pdftitle = {Bounded{ }mean{ }oscillation{ }function},
		pdfsubject = {Harmonic{ }analysis}]{hyperref}
\hypersetup{colorlinks = true,
			citecolor = mallard,
			linkcolor = eggplant,
			urlcolor = dusky cerulean}
			
\newcommand{\enumerationisinitium}{\begin{enumerate}[nolistsep, wide, label = \textnormal{($\mathnormal{\arabic{*}}$)}, ref = \textnormal{($\mathnormal{\arabic{*}}$)}]}
\newcommand{\enumerationisfinis}{\end{enumerate}}
\newcommand{\subenumerationisinitium}{\begin{enumerate}[nolistsep, wide, label = \textnormal{(\roman{*})}, ref = \textnormal{(\roman{*})}]}
\newcommand{\subenumerationisfinis}{\end{enumerate}}

\newenvironment{indent paragraph: 15pt}{%
  \par%
  \leftskip=15pt%
  \noindent\ignorespaces}{%
  \par}

\newenvironment{indent paragraph: 30pt}{%
  \par%
  \leftskip=30pt%
  \noindent\ignorespaces}{%
  \par}

\makeatletter
\newcommand{\xMapsto}[2][]{\ext@arrow 0599{\Mapstofill@}{#1}{#2}}
\def\Mapstofill@{\arrowfill@{\Mapstochar\Relbar}\Relbar\Rightarrow}
\makeatother

\DeclareSymbolFont{tipa}{T3}{cmr}{m}{n}
\DeclareMathAccent{\invertedbreve}{\mathalpha}{tipa}{16}

\DeclareMathOperator{\Laplacian}{\bigtriangleup}

\DeclareRobustCommand{\reflectedepsilon}{\reflectbox{$\epsilon$}}

\newcommand{\Hardy}{H}

\newcommand{\Lebesgue}{L}

\newcommand{\viz}{\stackrel{\textnormal{\tiny{viz}}}{=}}

\makeatletter

\newcommand*\Neginternal[3]{\mathpalette\Neg@{{#1}{#2}{#3}}}
\newcommand*\Neg@[2]{\Neg@@{#1}#2}
\newcommand*\Neg@@[4]{%
  \mathrel{\ooalign{%
    $\m@th#1#4$\cr
    \hidewidth$\m@th#3{#1}\mkern\muexpr#2*2$\hidewidth\cr
  }}%
}

\newcommand*\negslash[1]{\m@th#1\not\mathrel{\phantom{=}}}
\newcommand*\snegslash[1]{\rotatebox[origin=c]{60}{$\m@th#1-$}}
\newcommand*\ssnegslash[1]{\rotatebox[origin=c]{60}{$\m@th#1{\dabar@}\mkern-7mu{\dabar@}$}}
\newcommand*\sssnegslash[1]{\rotatebox[origin=c]{60}{$\m@th#1\dabar@$}}
\makeatother 

\makeatletter
\newcommand\footnoteref[1]{\protected@xdef\@thefnmark{\ref{#1}}\@footnotemark}
\makeatother

\usepackage{accents}

\addto{\captionsenglish}{\renewcommand{\refname}{Bibliography · List of Works Mentioned in the Article}
	}

\let\OLDthebibliography\thebibliography
\renewcommand\thebibliography[1]
{\OLDthebibliography{#1}
\setlength{\parskip}{0pt}
\setlength{\itemsep}{0pt plus 0.3ex}}

\title[]{Bounded Mean Oscillation: \\ an $\protect\pseudobold{\mathbb{R}}$-Function with Multi-$\protect\pseudobold{\mathbb{K}_6}$ Cubes. \\
Dual of the Hardy Space $\protect\pseudobold{\Hardy^1}$ and Banach Extent}

\author{\small\href{https://edoardoniccolai.com}{Edoardo Niccolai}
	}

\begin{document}

\begin{abstract}
This paper investigates the concept of harmonic functions of bounded mean oscillation, starting from John–Nirenberg's pioneering studies, under a renewed formalism, suitable for bringing out some fundamental properties inherent in it. In more detail: after a quick introduction, the second Section presents the main theorem, plus complete proof, relating to this function; in the third Section there is a suggestion on the exponential integrability (theorem and sketch of proof), while the fourth Section deals with the duality of Hardy space $\Hardy^1$ and bounded mean oscillation, with some ideas for a demonstration. The writing closes with a graphic appendix.

\vspace{2mm}

\end{abstract}

\maketitle

\begingroup
\hypersetup{linktocpage}
\tableofcontents
\thispagestyle{empty}
\endgroup

\section{A Few Words of Introduction}

\emph{Bounded mean oscillation} indicates an $\mathbb{R}$-function  characterized by a finite mean oscillation. This can be expressed by a theorem, investigated by F. John and L. Nirenberg \cite{John and Nirenberg "On Functions of Bounded Mean Oscillation"} in 1961. Let us try to give a more modern and concise form to this theorem, so as to outline some of its essential peculiarities in a few quick strokes. We will also take a look at the integrability of this harmonic function, and especially at its function space.

\section{Bounded Mean Oscillation's Statement: How to do Harmonic Analysis with Cubes}

\begin{theorema}
\label{theorema "John–Nirenberg"}
Let $\varphi[\mathbb{R}]$ be a function of bounded mean oscillation, i.e. $\varphi \in \varphi_{\mathrm{b}\sim}^{[\mathbb{R}]}$, $\mathbb{R} \viz \mathbb{R}^n$, for any cube $\mathbb{K}_6$, and multi-index $\alpha = (\alpha_1, \mathellipsis, \alpha_n) \in \mathbb{Z}_* = \{0\} \cup \mathbb{Z}_+$ of order $|\alpha| \leqslant k \in \mathbb{Z}_*$, where $\alpha > 0$ and $|\alpha| = \alpha_1 + \cdots + \alpha_n$. One has the inequality
\begin{equation}
\label{equation "Theorematic inequality"}
	\bigg|\left(x \in \mathbb{K}_6 \colon \left|\varphi(x) - N_{(\mathrm{a})}^{\mathbb{K}_6}\varphi\right| > \alpha\right)\bigg| \leqslant e |\mathbb{K}_6|\exp{\left\{-\frac{(2^ne)^{-1}\alpha}{\|\varphi\|_{\mathrm{b}\sim}}\right\}},
\end{equation}
in which $N_{(\mathrm{a})}^{\mathbb{K}_6}$ connotes the average value of $\varphi$ over $\mathbb{K}_6$, and $\|\varphi\|_{\mathrm{b}\sim} \viz \|\varphi\|_\textsc{bmo}$ (\textsc{bmo} is, patently, for \emph{bounded mean oscillation}).
\end{theorema}

\begin{proof}
We start from a cube $\invertedbreve{\mathbb{K}}_6$, and we choose this principle: 
\begin{equation}
\label{equation "principle for invertedbreveK6 1"}
	\frac{1}{|\invertedbreve{\mathbb{K}}_6|}\int_{\invertedbreve{\mathbb{K}}_6}\left|\varphi(x) - N_{(\mathrm{a})}^{\mathbb{K}_6}\varphi\right|dx > \reflectedepsilon,	
\end{equation}
with a constant $\reflectedepsilon > 1$. But immediately we realize that the principle \eqref{equation "principle for invertedbreveK6 1"} does not apply to $\mathbb{K}_6$, and in fact: $\frac{1}{|\mathbb{K}_6|}\int_{\mathbb{K}_6}\left|\varphi(x) - N_{(\mathrm{a})}^{\mathbb{K}_6}\varphi\right|dx \leqslant \|\varphi\|_{\mathrm{b}\sim} = 1 < \reflectedepsilon$. 

 We remember that $|\mathbb{K}_6|$ is the volume of $\mathbb{K}_6$, which translates into its Lebesgue measure, $|\mathbb{K}_6| \viz \bblambda(\mathbb{K}_6)$; idem for $|\invertedbreve{\mathbb{K}}_6| \viz \bblambda(\invertedbreve{\mathbb{K}}_6)$, or for any other cube $|\mathbb{K}_{(\cdot)}| \viz \bblambda(\mathbb{K}_{(\cdot)})$, plainly.

The next step is to take a cube $\mathbb{K}_6^{\textgreek{\textit{α}}} \viz \mathbb{K}_6$, and divide it into $2^n$ equal subcubes, so as to see if the resulting subcubes, which we can call $\invertedbreve{\mathbb{K}}_6^{\textgreek{\textit{{α}}}_n}$, satisfy the principle \eqref{equation "principle for invertedbreveK6 1"}. Such a process is repeated an infinite number of times, in order to gain a multiset of cubes $\left\{\mathbb{K}_{6_r}^{\textgreek{\textit{β}}}\right\}_r$ whose features are:
\begin{subequations}
\label{subequations "C6r features"}
\begin{align}
	& 
	\label{align "feature 1"}
	\left|\varphi - N_{(\mathrm{a})}^{\mathbb{K}_6^{\textgreek{\textit{α}}}}\varphi\right| \leqslant \reflectedepsilon \text{ almost everywhere on }\mathbb{K}_6^{\textgreek{\textit{α}}}\backslash\bigcup_r\mathbb{K}_{6_r}^{\textgreek{\textit{β}}}, \\
	& \left|N_{(\mathrm{a})}^{\mathbb{K}_{6_r}^{\textgreek{\textit{β}}}}\varphi - N_{(\mathrm{a})}^{\mathbb{K}_6^{\textgreek{\textit{α}}}}\varphi\right| \leqslant 2^n\reflectedepsilon, \\
	& \reflectedepsilon < \left|\mathbb{K}_{6_r}^{\textgreek{\textit{β}}}\right|^{-1}\int_{\mathbb{K}_{6_r}^{\textgreek{\textit{β}}}}\left|\varphi(x) - N_{(\mathrm{a})}^{\mathbb{K}_6^{\textgreek{\textit{α}}}}\varphi\right|dx \leqslant 2^n \reflectedepsilon, \\
	& \sum_r\left|\mathbb{K}_{6_r}^{\textgreek{\textit{β}}}\right| \leqslant 1/\reflectedepsilon\sum_r\int_{\mathbb{K}_{6_r}^{\textgreek{\textit{β}}}}\left|\varphi(x) - N_{(\mathrm{a})}^{\mathbb{K}_6^{\textgreek{\textit{α}}}}\varphi\right|dx \leqslant 1/\reflectedepsilon\left|\mathbb{K}_6^{\textgreek{\textit{α}}}\right|.
\end{align}
\end{subequations}
Remember that the interior of each $\mathbb{K}_{6_r}^{\textgreek{\textit{β}}}$ is included in $\mathbb{K}_6^{\textgreek{\textit{α}}}$. The symbolic apparatus \eqref{subequations "C6r features"} allows to hybridize the principle \eqref{equation "principle for invertedbreveK6 1"} with this new principle: \\
\begin{equation}
\label{equation "equation "principle for invertedbreveK6 2"}
	\frac{1}{|\invertedbreve{\mathbb{K}}_6|}\int_{\invertedbreve{\mathbb{K}}_6}\left|\varphi(x) - N_{(\mathrm{a})}^{\mathbb{K}_{6_r}^{\textgreek{\textit{β}}}}\varphi\right|dx > \reflectedepsilon, 
\end{equation}
which however is not satisfied by $\mathbb{K}_{6_r}^{\textgreek{\textit{β}}}$. The process of decomposition into $2^n$ equal subcubes mentioned above must therefore be replicated for $\mathbb{K}_{6_r}^{\textgreek{\textit{β}}}$, with the aim that the subcubes obtained can satisfy the new principle, thereby generating a multiset of cubes $\left\{\mathbb{K}_{6_s}^{\textgreek{\textit{γ}}}\right\}_s$, which are contained in $\mathbb{K}_{6_r}^{\textgreek{\textit{β}}}$. This new subdivision satisfies any feature in \eqref{subequations "C6r features"}, when $\mathbb{K}_6^{\textgreek{\textit{α}}}$ is replaced with $\mathbb{K}_6^{\textgreek{\textit{β}}}$, and $\mathbb{K}_6^{\textgreek{\textit{β}}}$ with $\mathbb{K}_6^{\textgreek{\textit{γ}}}$.

So let us proceed with a further principle, which is identical to that in \eqref{equation "equation "principle for invertedbreveK6 2"}, except that the average changes, that is 
\[
	N_{(\mathrm{a})}^{\mathbb{K}_{6_s}^{\textgreek{\textit{γ}}}}.
\]
Consequently, we get a multiset of cubes $\left\{\mathbb{K}_{6_t}^{\textgreek{\textit{δ}}}\right\}_t$, each of which is contained in $\mathbb{K}_{6_s}^{\textgreek{\textit{γ}}}$. Here is the up-to-date list of features, with the formation of two sets of cubes, this time denoted by $\mathbb{K}_{6_r}^{\textgreek{\textit{λ}}}$:
\begin{subequations}
\begin{align}
	&
	\label{align "feature 1.1"} 
	\left|\varphi - N_{(\mathrm{a})}^{\mathbb{K}_{6_{\dot{r}}}^{\textgreek{\textit{λ}}- 1}}\varphi\right| \leqslant \reflectedepsilon \text{ almost everywhere on }\mathbb{K}_{6_{\dot{r}}}^{\textgreek{\textit{λ}}- 1}\backslash\bigcup_r\mathbb{K}_{6_r}^{\textgreek{\textit{λ}}}, \\
	& \left|N_{(\mathrm{a})}^{\mathbb{K}_{6_r}^{\textgreek{\textit{λ}}}}\varphi - N_{(\mathrm{a})}^{\mathbb{K}_{6_{\dot{r}}}^{\textgreek{\textit{λ}}- 1}}\varphi\right| \leqslant 2^n\reflectedepsilon, \\
	& \reflectedepsilon < \left|\mathbb{K}_{6_r}^{\textgreek{\textit{λ}}}\right|^{-1}\int_{\mathbb{K}_{6_r}^{\textgreek{\textit{λ}}}}\left|\varphi(x) - N_{(\mathrm{a})}^{\mathbb{K}_{6_{\dot{r}}}^{\textgreek{\textit{λ}}- 1}}\varphi\right|dx \leqslant 2^n \reflectedepsilon, \\
	& 
	\label{align "feature 4.1"}
	\sum_r\left|\mathbb{K}_{6_r}^{\textgreek{\textit{λ}}}\right| \leqslant 1/\reflectedepsilon\sum_{\dot{r}}\left|\mathbb{K}_{6_{\dot{r}}}^{\textgreek{\textit{λ}}- 1}\right|.
\end{align}	
\end{subequations}
To follow some things to know (the rest is easily verifiable).
\subenumerationisinitium
\item The interior of each $\mathbb{K}_{6_r}^{\textgreek{\textit{λ}}}$ is included in some $\mathbb{K}_{6_{\dot{r}}}^{\textgreek{\textit{λ}} - 1}$.
\item The feature \eqref{align "feature 1.1"} is demonstrable via Lebesgue differentiation theorem \cite{Lebesgue "Lecons sur l'integration et la recherche des fonctions primitives"} \cite{Lebesgue "Sur l'integration des fonctions discontinues"}: all points in $\mathbb{K}_{6_{\dot{r}}}^{\textgreek{\textit{λ}} - 1}\backslash\bigcup_r\mathbb{K}_{6_r}^{\textgreek{\textit{λ}}}$ are associated with a succession of cubes that can be shrunk up to any spatial punctuality, subsequently the averages over each cube is at most $\reflectedepsilon$.
\item The feature \eqref{align "feature 4.1"} is provable with this formula:
\begin{align}
	\sum_r\left|\mathbb{K}_{6_r}^{\textgreek{\textit{λ}}}\right| < 1/\reflectedepsilon\sum_r\int_{\mathbb{K}_{6_r}^{\textgreek{\textit{λ}}}}\left|\varphi(x) - N_{(\mathrm{a})}^{\mathbb{K}_{6_{\dot{r}}}^{\textgreek{\textit{λ}} - 1}}\varphi\right|dx & = 1/\reflectedepsilon\sum_{\dot{r}}\int_{\mathbb{K}_{6_r}^{\textgreek{\textit{λ}}}}\left|\varphi(x) - N_{(\mathrm{a})}^{\mathbb{K}_{6_{\dot{r}}}^{\textgreek{\textit{λ}} - 1}}\varphi\right|dx \notag \\
	& \leqslant 1/\reflectedepsilon\sum_{\dot{r}}\int_{\mathbb{K}_{6_{\dot{r}}}^{\textgreek{\textit{λ}} - 1}}\left|\varphi(x) - N_{(\mathrm{a})}^{\mathbb{K}_{6_{\dot{r}}}^{\textgreek{\textit{λ}} - 1}}\varphi\right|dx \notag \\
	& \leqslant 1/\reflectedepsilon\sum_{\dot{r}}\left|\mathbb{K}_{6_{\dot{r}}}^{\textgreek{\textit{λ}} - 1}\right|\|\varphi\|_{\mathrm{b}\sim} \notag \\
	& = 1/\reflectedepsilon\sum_{\dot{r}}\left|\mathbb{K}_{6_{\dot{r}}}^{\textgreek{\textit{λ}} - 1}\right|.
\end{align}
\subenumerationisfinis
Accordingly, from \eqref{align "feature 4.1"} one achieves $\sum_r\left|\mathbb{K}_{6_r}^{\textgreek{\textit{λ}}}\right| \leqslant \reflectedepsilon^{-{\textgreek{\textit{λ}}}}\left|\mathbb{K}_6^{\textgreek{\textit{α}}}\right|$. We are now able to write these inequalities:
\begin{subequations}
\begin{align}
	& \left|N_{(\mathrm{a})}^{\mathbb{K}_{6_r}^{\textgreek{\textit{β}}}}\varphi - N_{(\mathrm{a})}^{\mathbb{K}_6^{\textgreek{\textit{α}}}}\varphi\right| \leqslant 2^n\reflectedepsilon, \\
	& \left|\varphi - N_{(\mathrm{a})}^{\mathbb{K}_{6_r}^{\textgreek{\textit{β}}}}\varphi\right| \leqslant \reflectedepsilon \text{ almost everywhere on }\mathbb{K}_{6_r}^{\textgreek{\textit{β}}}\backslash\bigcup_s\mathbb{K}_{6_s}^{\textgreek{\textit{γ}}}, \\
	& \text{ergo } \left|\varphi - N_{(\mathrm{a})}^{\mathbb{K}_6^{\textgreek{\textit{α}}}}\varphi\right| \leqslant 2^n\reflectedepsilon + \reflectedepsilon \text{ almost everywhere on }\mathbb{K}_{6_r}^{\textgreek{\textit{β}}}\backslash\bigcup_s\mathbb{K}_{6_s}^{\textgreek{\textit{γ}}}, \\
	& 
	\label{align "inequality 4"}
	\text{and, from } \eqref{align "feature 1"}, \left|\varphi - N_{(\mathrm{a})}^{\mathbb{K}_6^{\textgreek{\textit{α}}}}\varphi\right| \leqslant 2^n(2\reflectedepsilon) \text{ almost everywhere on }\mathbb{K}_6^{\textgreek{\textit{α}}}\backslash\bigcup_s\mathbb{K}_{6_s}^{\textgreek{\textit{γ}}}.
\end{align}
\end{subequations}
Then
\begin{subequations}
\begin{align}
	& \left|\varphi - N_{(\mathrm{a})}^{\mathbb{K}_{6_s}^{\textgreek{\textit{γ}}}}\varphi\right| \leqslant \reflectedepsilon \text{ almost everywhere on }\mathbb{K}_{6_s}^{\textgreek{\textit{γ}}}\backslash\bigcup_t\mathbb{K}_{6_t}^{\textgreek{\textit{δ}}}, \\
	& \text{ergo } \left|\varphi - \mathbb{K}_6^{\textgreek{\textit{α}}}\varphi\right| \leqslant 2^n(3\reflectedepsilon) \text{ almost everywhere on }\mathbb{K}_{6_s}^{\textgreek{\textit{γ}}}\backslash\bigcup_t\mathbb{K}_{6_t}^{\textgreek{\textit{δ}}}, 
\end{align}
\end{subequations}
once we have combined $\left|N_{(\mathrm{a})}^{\mathbb{K}_{6_s}^{\textgreek{\textit{γ}}}}\varphi - N_{(\mathrm{a})}^{\mathbb{K}_{6_{\dot{s}}}^{\textgreek{\textit{β}}}}\varphi\right| \leqslant 2^n\reflectedepsilon$ and $\left|N_{(\mathrm{a})}^{\mathbb{K}_{6_{\dot{s}}}^{\textgreek{\textit{β}}}}\varphi - N_{(\mathrm{a})}^{\mathbb{K}_6^{\textgreek{\textit{α}}}}\varphi\right| \leqslant 2^n\reflectedepsilon$.

The inequality \eqref{align "inequality 4"} permits to extend the procedure also on $\mathbb{K}_6^{\textgreek{\textit{α}}}\backslash\bigcup_t\mathbb{K}_{6_t}^{\textgreek{\textit{δ}}}$, from which we derive a further relation between different values: 
\begin{equation}
\label{equation "Further relation between different values"}
	\left|\varphi - N_{(\mathrm{a})}^{\mathbb{K}_6^{\textgreek{\textit{α}}}}\varphi\right| \leqslant 2^n({\textgreek{\textit{λ}}}\reflectedepsilon) \text{ almost everywhere on }\mathbb{K}_6^{\textgreek{\textit{α}}}\backslash\bigcup_t\mathbb{K}_{6_t}^{\textgreek{\textit{λ}}}. 
\end{equation}	
The expression \eqref{equation "Further relation between different values"} gives us the opportunity to define this containment:
\begin{equation}
	\bigg|\left(x \in \mathbb{K}_6 \colon \left|\varphi(x) - N_{(\mathrm{a})}^{\mathbb{K}_6}\varphi\right| > 2^n({\textgreek{\textit{λ}}}\reflectedepsilon)\right)\bigg| \subset \bigcup_r\mathbb{K}_{6_r}^{\textgreek{\textit{λ}}}.
\end{equation}

The way to prove the theorematic inequality \eqref{equation "Theorematic inequality"} is to go through the inequality 
\begin{equation}
	\sum_r\left|\mathbb{K}_{6_r}^{\textgreek{\textit{λ}}}\right| \leqslant \reflectedepsilon^{-{\textgreek{\textit{λ}}}}\left|\mathbb{K}_6^{\textgreek{\textit{α}}}\right|,
\end{equation}
previously encountered, and the one present in the relation \eqref{equation "Further relation between different values"}. Letting first $2^n({\textgreek{\textit{λ}}}\reflectedepsilon) < \alpha \leqslant 2^n(\textgreek{\textit{λ}} + 1)\reflectedepsilon$, for $\alpha > 0$ and $\textgreek{\textit{λ}} \geqslant 0$, and defining then $-\textgreek{\textit{λ}} \leqslant 1 - \alpha/2^n\reflectedepsilon$, the final formula, once it is determined that $\reflectedepsilon = e > 1$, is
\begin{align}
	\bigg|\left(x \in \mathbb{K}_6 \colon \left|\varphi - N_{(\mathrm{a})}^{\mathbb{K}_6}\varphi\right| > \alpha\right)\bigg| & \leqslant \bigg|\left(x \in \mathbb{K}_6 \colon \left|\varphi - N_{(\mathrm{a})}^{\mathbb{K}_6}\varphi\right| > 2^n({\textgreek{\textit{λ}}}\reflectedepsilon)\right)\bigg|, \notag \\
	& \leqslant \sum_r\left|\mathbb{K}_{6_r}^{\textgreek{\textit{λ}}}\right| \leqslant 1/\reflectedepsilon^{\textgreek{\textit{λ}}}\left|\mathbb{K}_6^{\textgreek{\textit{α}}}\right| \\
	& = \left|\mathbb{K}_6\right|\exp{\{-\textgreek{\textit{λ}}\log\reflectedepsilon\}} \leqslant \left|\mathbb{K}_6\right|\reflectedepsilon\exp{\left\{-\frac{\alpha\textgreek{\textit{λ}}\log\reflectedepsilon}{2^n\reflectedepsilon}\right\}},
\end{align} 
as required.
\end{proof}

\section{Cubic Exponential Integrability}

The Theorem \ref{theorema "John–Nirenberg"} implies that, for any $\varphi \in \varphi_{\mathrm{b}\sim}^{[\mathbb{R}]}$, there is  an exponential integrability over all $\mathbb{K}_6$, i.e.
\begin{equation}
	\frac{1}{\left|\mathbb{K}_6\right|}\int_{\mathbb{K}_6}\exp{\left\{\frac{\zeta\left|\varphi(x) - N_{(\mathrm{a})}^{\mathbb{K}_6}\varphi\right|}{\|\varphi\|_{\mathrm{b}\sim}}\right\}}dx \leqslant 1 - \frac{2^ne^2\zeta}{2^ne\zeta - 1},
\end{equation} 
for some $\zeta < \frac{1}{2^ne}$. Which is proved by the resulting combination,
\begin{equation}
	\frac{1}{\left|\mathbb{K}_6\right|}\int_{\mathbb{K}_6}\exp{\left\{\frac{\zeta\left|\varphi(x) - N_{(\mathrm{a})}^{\mathbb{K}_6}\varphi\right|}{\|\varphi\|_{\mathrm{b}\sim}}\right\}}dx \leqslant \int^\infty_0e^\alpha\exp{\left\{-\frac{(2^ne)^{-1}\left(\frac{\alpha}{\zeta}\|\varphi\|_{\mathrm{b}\sim}\right)}{\|\varphi\|_{\mathrm{b}\sim}}\right\}}d\alpha = c_{(\zeta)},
\end{equation}
for $\zeta < (2^ne)^{-1}$, where $c_{(\zeta)}$ is a constant.

\section[Function Space of $\varphi_{\mathrm{b}\sim}^{[\mathbb{R}]}$]{Function Space of $\protect\pseudobold{\varphi_{\mathrm{b}\sim}^{[\mathbb{R}]}}$}

Now let us talk about space in relation to our harmonic function.

\subsection[Dual of the Hardy Space $\Hardy^1(\mathbb{R}^n)$ and Banach Extent]{Dual of the Hardy Space $\protect\pseudobold{\Hardy^1(\mathbb{R}^n)}$ and Banach Extent}

\enumerationisinitium
\item As F. John \cite{John "Rotation and strain"} pointed out, the space of a function of bounded mean oscillation is but a (Lebesgue) function space. C. Fefferman \cite{Fefferman "Characterizations of bounded mean oscillation"}, and later E.M. Stein \cite[I, IV]{Fefferman and Stein Hp spaces of several variables"}, have shown that there is a duality between the Hardy space $\Hardy^{p = 1}$ and all bounded mean oscillation functions—which gives a characterization of $\Hardy^p$ in terms of boundary properties of harmonic functions.\footnote{
	{} A Hardy (or Hardy–Freidrich) space $\Hardy^p(\mathbb{R}^n)$, $0 < p < \infty$, more correctly defined as a real Hardy space on $\mathbb{R}^n$, is a space of distributions on the real line, whose singularity increases as $p$ decreases. One of its most significant forms is that of space of all tempered distributions:
\[
	\|\varphi\|_{\Hardy^p} = \left\|\left(\sum_{r \in \mathbb{Z}}\left|\Laplacian_r(\varphi)\right|^2\right)^\frac{1}{2}\right\|_{\Lebesgue^p} < \infty, \enspace 0 < p \leqslant 1.
\]	
	Cfr. G.H. Hardy \cite{Hardy "The Mean Value of the Modulus of an Analytic Function"} and F. Freidrich \cite{Freidrich "Uber die Randwerte einer analytischen Funktion"}.
	}  
Put otherwise, 
\subenumerationisinitium
\item a bounded mean oscillation $\varphi_{\mathrm{b}\sim}$ is the dual of the Hardy space $\Hardy^1(\mathbb{R}^n)$,
\item a continuous linear functional, say, $\mathscr{I}_\textsc{l}$ on $\Hardy^1(\mathbb{R}^n)$ is realizable in terms of a mapping:
\begin{equation}
\label{equation "Continuous linear functional"}
	\mathscr{I}_\textsc{l}[\eta_h] = \int_{\mathbb{R}^n}\varphi(x)\eta_h(x)dx,
\end{equation}
where $\varphi \in \varphi_{\mathrm{b}\sim}^{[\mathbb{R}]}$, and $\eta_h \in \Hardy^1$. From Eq. \eqref{equation "Continuous linear functional"} it is inferred that, for each $\varphi \in \varphi_{\mathrm{b}\sim}^{[\mathbb{R}]}$, the functional $\mathscr{I}_\textsc{l}[\eta_h]$ is bounded on $\Hardy^1$ via an inequality of this kind:
\begin{equation}
\label{equation "Inequality for duality between Hardy space and function of bounded mean oscillation"}
	\left|\int_{\mathbb{R}^n}\varphi(\eta_h)dx\right| \leqslant c \left\|\varphi\right\|_{\mathrm{b}\sim}\|\eta_h\|_{\Hardy^1},
\end{equation}
setting $\eta_h$ in the Hardy subspace $\Hardy^1_\tau \subset \Hardy^1$. And hence it is useful to write that
\begin{equation}
	\int_{\mathbb{R}^n}\varphi(\eta_h)dx = \sum_j\kappa_j\int_{\mathbb{R}^n}\varphi(x)\tau_j(x)dx,
\end{equation}
clarified the detail that $\eta_h = \sum\kappa_j\tau_j$, and that $\eta_h$ is convergent in $\Lebesgue^1(\mathbb{R}^n)$. The demonstration of inequality \eqref{equation "Inequality for duality between Hardy space and function of bounded mean oscillation"}, for a bounded $\varphi$ and for any $\eta_h \in \Hardy^1$,
 is finally given by
\begin{equation}
	\left|\int\varphi(x)\eta_h(x)dx\right| \leqslant \sum_j \left|\frac{\kappa_j}{\mathbb{B}_j}\right|\int_{\mathbb{B}_j}\left|\varphi(x) - \varphi_{\mathbb{B}_j}\right|dx \leqslant \sum |\kappa_j| \cdot \|\varphi\|_{\mathrm{b}\sim},
\end{equation}
given a ball $\mathbb{B}_j \in \mathbb{R}^n$, once the non-equality $\left|\tau_j(x)\right| \leqslant \left|\mathbb{B}_j\right|^{-1}$ is designated.
\subenumerationisfinis
\item It is possible to translate everything described above into a simple assertion: $\Hardy^1(\mathbb{R}^n)$ can be considered as a replacement for $\Lebesgue^1(\mathbb{R}^n)$, whilst $\|\varphi\|_{\mathrm{b}\sim}$ plays the same role with respect to the space $\Lebesgue^\infty(\mathbb{R}^n)$ of bounded functions on $\mathbb{R}^n$.
\item A bounded mean oscillation function (supported in a cube $\mathbb{K}_6$ on $\mathbb{R}^n$) is locally $\Lebesgue^p$-integrable,
\begin{equation}
	\|\varphi\|_{\mathrm{b}\sim} = \sup_{\mathbb{K}_6}\left\{\frac{1}{|\mathbb{K}_6|}\int_{\mathbb{K}_6}\left|\varphi(x) - N_{(\mathrm{a})}^{\mathbb{K}_6}\varphi\right|dx\right\} < \infty, \enspace 1 < p < \infty.
\end{equation}
Which means that $\|\varphi\|_{\mathrm{b}\sim}$ is the Banach space \cite{Banach "Sur les operations dans les ensembles abstraits et leur application aux equations integrales"} of every function $\varphi \in \Lebesgue^1_\mathrm{loc}(\mathbb{R}^n)$.
\enumerationisfinis

\setcounter{secnumdepth}{0}
\section{Graphic Appendix. Cubic Inclusion of Matryoshka-type}
\setcounter{secnumdepth}{3}

When it comes to cubes contained within other cubes, the simplified graphic representation is something like this,

\vspace{2mm}

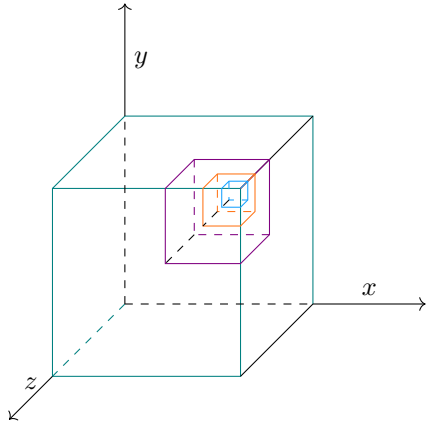
\begingroup
\begin{center}
\begin{tikzpicture}
  \draw[draw = mallard] (2.5, 0, 0) coordinate (x) |- (0, 2.5, 0) coordinate [midway] (h) coordinate (y) -- (0, 2.5, 2.5) coordinate (a) -- (0, 0, 2.5) coordinate (z) -- (2.5, 0, 2.5) edge (x) -- (2.5, 2.5, 2.5) coordinate (v) edge (h)
  -- (a);
  \draw[dashed, draw = mallard] (0, 0 ,0) coordinate (o) edge (x) edge (y) -- (z);
  \draw[->] (x) -- +(1.5pt, 0, 0) node [midway, above]{$x$};
  \draw[->] (y) -- +(0, 1.5pt, 0) node [midway, right]{$y$};
  \draw[->] (z) -- +(0, 0, 1.5pt) node [midway, above]{$z$};
  \draw[draw = eggplant] (v) -- ++(0, 0, -1) coordinate (d) -- ++(-1, 0, 0) coordinate (e) -- ++(0, 0, 1) |- ++(1, -1, 0) coordinate [midway] (f) -- ++(0, 0, -1) coordinate (g) -- (d);
  \draw[dashed, draw = eggplant] (e) -- ++(0, -1, 0) coordinate (c) edge (f) -- (g);
  \draw[draw = pumpkin] (v) -- ++(0, 0, -0.5) coordinate (d) -- ++(-0.5, 0, 0) coordinate (e) -- ++(0, 0, 0.5) |- ++(0.5, -0.5, 0) coordinate [midway] (f) -- ++(0, 0, -0.5) coordinate (g) -- (d);
  \draw[dashed, draw = pumpkin] (e) -- ++(0, -0.5, 0) coordinate (c) edge (f) -- (g);
  \draw[draw = cyan-blue] (v) -- ++(0, 0, -0.25) coordinate (d) -- ++(-0.25, 0, 0) coordinate (e) -- ++(0, 0, 0.25) |- ++(0.25, -0.25, 0) coordinate [midway] (f) -- ++(0, 0, -0.25) coordinate (g) -- (d);
  \draw[dashed, draw = cyan-blue] (e) -- ++(0, -0.25, 0) coordinate (c) edge (f) -- (g);
\end{tikzpicture}
\end{center}
\endgroup
with a inclusion of Matryoshka-type: $\textcolor{mallard}{\{}
	\textcolor{mallard}{\mathbb{K}_6},
	\textcolor{eggplant}{\{}\textcolor{eggplant}{\mathbb{K}_6},
	\textcolor{pumpkin}{\{}\textcolor{pumpkin}{\mathbb{K}_6}, \textcolor{cyan-blue}{\{}\textcolor{cyan-blue}{\mathbb{K}_6}\textcolor{cyan-blue}{\}}\textcolor{pumpkin}{\}}\textcolor{eggplant}{\}
	\textcolor{mallard}{\}}}$, where any $\mathbb{K}_6 = \{\varnothing\}$.

\setcounter{footnote}{0} 

\cleardoublepage
\phantomsection
\addcontentsline{toc}{section}{\refname}



\end{document}